\documentclass[11pt,letter]{article}
\usepackage[margin=1in]{geometry}

\usepackage[round]{natbib}

\usepackage{amsmath}
\usepackage{amssymb}
\usepackage{amsfonts}
\usepackage{amscd}

\usepackage{amssymb}
\usepackage{amsthm}
\usepackage{enumitem}
\usepackage{xpatch}
\usepackage{version,xspace}

\usepackage{enumitem}
\setlist{nosep}

\newtheorem{theorem}{Theorem}[section]
\newtheorem{corollary}[theorem]{Corollary}

\newtheorem{proposition}[theorem]{Proposition}
\newtheorem{remark}[theorem]{Remark}

{
      \theoremstyle{plain}

}

\newcommand{\subscr}[2]{#1_{\textup{#2}}}

\newcommand{\R}{\mathbb{R}}

\newcommand\norm[1]{\left\lVert#1\right\rVert}
\renewcommand{\norm}[1]{\|#1\|}

\newcommand{\inpr}[2]{\langle#1,#2\rangle}

\DeclareSymbolFont{bbold}{U}{bbold}{m}{n}
\DeclareSymbolFontAlphabet{\mathbbold}{bbold}

\newcommand{\ones}{1}
\newcommand{\zeros}{0}

\newcommand{\SuL}{\mathcal{S}^{1}_{\mu,L}}

\newcommand{\fAcc}{\subscr{F}{Acc}}

\newcounter{saveenum}

\title{A Contraction Theory Approach to Optimization Algorithms from Acceleration Flows}

\author{Pedro Cisneros-Velarde, Francesco Bullo
  \thanks{Pedro Cisneros-Velarde (pacisne@gmail.com) [corresponding author] is with the University of Illinois at Urbana-Champaign and Francesco Bullo (bullo@ucsb.edu) is with the 
    University of California,
    Santa Barbara.}
    }
\date{}

\begin{document}
\maketitle

\begin{abstract}
Much recent interest has focused on the design of optimization algorithms
from the discretization of an associated optimization flow, i.e., a system
of differential equations (ODEs) whose trajectories solve an associated
optimization problem.
Such a design approach poses an important problem: how to find a principled
methodology to design and discretize appropriate ODEs.
This paper aims to provide a solution to this problem through the use of
contraction theory. 
We first introduce general mathematical results that explain how
contraction theory guarantees the stability of the implicit and explicit
Euler integration methods.
Then, we propose a novel system of ODEs, namely the
Accelerated-Contracting-Nesterov flow, and use contraction theory to
establish it is an optimization flow with exponential convergence rate,
from which the linear convergence rate of its associated optimization
algorithm is immediately established. Remarkably, a simple explicit Euler
discretization of this flow corresponds to the Nesterov acceleration
method.
Finally, we present how our approach leads to performance guarantees in the
design of optimization algorithms for time-varying optimization problems.
\end{abstract}

\section{Introduction}

\paragraph*{Problem statement and motivation}

There has been a recent interest in studying systems of ODEs that solve an optimization problem --– also known as \emph{optimization flows} --– with the understanding that their study can lead to the analysis and design of discrete-time solvers of optimization problems – also known as \emph{optimization algorithms}. This interest is motivated by the fact that analyzing a system of ODEs can be much simpler than analyzing a discrete system. 
Indeed, the ambitious goal of this research area 
is to find a “general theory mapping properties of ODEs into corresponding properties for discrete updates" --- as quoted from the seminal work~\citep{WS-SB-EJC:16}. Our paper aims to provide a solution to this problem. 
Ideally, the desired pipeline is to first design an optimization flow --– using all the machinery of dynamical systems analysis --– with good stability and convergence properties, and then formulate a principled way of guaranteeing such good properties translate to its associated optimization algorithm through discretization. 

A first problem in the 
literature 
is that the analysis of the optimization algorithm is commonly done separately or independently from the analysis of its associated optimization flow (e.g., see~\citep{AW-ACW-MIJ:16,ACW-BR-MIJ:18,BS-SSD-MIJ-WJS:18,JZ-AM-SS-AJ:18,BS-SSD-WS-MIJ:19,MM-MJ:19}), instead of the former analysis following directly as a consequence of the latter one. 
For example, separate Lyapunov analyses have been made for optimization flows and their associated algorithms.
This problem diminishes one of the very first motivations of analyzing a system of ODEs, namely, that its analysis should directly   
establish 
properties 
of its associated discretization. 

Moreover, it is highly important to formulate an optimization flow with good convergence properties so that its discretization may also lead to an 
optimization algorithm maintaining such good convergence properties,
without becoming 
numerically unstable~\citep{AW-ACW-MIJ:16,BS-SSD-WS-MIJ:19}. A long-standing objective is to find \emph{rate-matching} discretizations, i.e., discretizations that preserve the good convergence rate performance from the optimization flow to the optimization algorithm. 

Our paper aims to propose a solution to the aforementioned two problems by using \emph{contraction theory} and by proposing the Accelerated-Contracting-Nesterov optimization flow, respectively.  
Contraction theory provides a principled approach for directly establishing the convergence properties of the optimization algorithm directly from the analysis of its associated optimization flow. For example, our approach 
directly translates the exponential convergence rate of the proposed optimization flow to a linear convergence rate for the optimization algorithm.

We also mention that, to the best of our knowledge, no previous work in the literature relevant to the design of (discrete-time) optimization algorithms from optimization flows has considered the case where the objective function associated to the optimization problem is \emph{time-varying}, i.e., where the optimal solutions vary through time and form a trajectory. In this setting, the objective 
is to design an optimization flow 
and 
optimization algorithms that are able to track the time-varying optimal solution up to some bounded error. 
In this paper, we use contraction theory to directly establish the tracking performance of our proposed optimization flow and its associated optimization algorithms. Time-varying optimization has found diverse applications in machine learning, signal processing, robotics, etc.; e.g., see~\citep{AS-EDA-SP-GL-GBG:20} and the references therein. 
We must mention the existence of \emph{prediction-correction methods}~\citep{NB-AS-RC:20} which consist  on sampling a continuous-time time-varying optimization problem on fixed intervals of time and then finding approximate solutions to the full optimization problem within the sampling intervals. These methods are not within our problem's scope because we are interested in obtaining optimization algorithms whose dynamics directly track the solution trajectory of the problem.

\paragraph*{Literature review}

There is a large 
growing literature on the analysis and design of optimization algorithms based on ODEs, so we limit ourselves to mentioning some representative works. 
The seminal works~\citep{WS-SB-EJC:16,BS-SSD-MIJ-WJS:18} propose systems of ODEs based on continuous approximations or versions of the Heavy-ball and the Nesterov acceleration methods, providing a thorough analysis of these optimization flows with the purpose of gaining more insights about the momentum or acceleration phenomenon in their discrete-time counterparts. 
Examples of works that 
design optimization algorithms by direct discretization of optimization flows include: the use of optimization flows derived from variational approaches~\citep{AW-ACW-MIJ:16}, the use of both explicit and implicit Euler discretizations~\citep{ACW-BR-MIJ:18,BS-SSD-WS-MIJ:19}, the use of Runge-Kutta integrators~\citep{JZ-AM-SS-AJ:18}, the use of semi-implicit Euler integration~\cite{MM-MJ:19}, the use of opportunistic state triggering~\citep{MV-JC:19}, the use of symplectic methods~\citep{GF-JS-DPR-RV:20,BS-SSD-WS-MIJ:19}. 
Most of the previous works are based on Lyapunov analysis, as well as on the discretization of the 
ODEs  
proposed in the works~\citep{WS-SB-EJC:16,BS-SSD-MIJ-WJS:18} or a generalization of them.

To the best of our knowledge, 
only the approaches based on symplectic integration and opportunistic state triggering allows for a principled way of discretizing optimization flows 
while preserving good convergence properties 
and without the need to do an independent analysis of the discretizations. In this way, these works are closer in spirit to the approach in this paper;
however, in contrast to them, our paper uses the classic 
implicit and explicit Euler integration methods.

Contraction theory is a mathematical tool that analyzes the \emph{incremental stability} of dynamical systems, i.e, whether two different trajectories of a dynamical system converge or at least do not diverge towards each other, and has a long history of development in the control theory community
~\citep{WAC:1965,MV:02,WL-JJES:98,IRM-JJES:17b}. An introduction and survey can be found in~\citep{ZA-EDS:14b}. 
Contraction theory has been used in the analysis of optimization flows associated to certain classes of constrained optimization problems~\citep{HDN-TLV-KT-JJS:18,PCV-SJ-FB:19r}, distributed systems~\citep{NMB-JJES:20}, and connections with gradient flows has been made~\citep{PMW-JJES:20}. Parallel to this development, an important concept that implies contraction, known as the one-sided Lipschitz condition, has been used in numerical analysis
~\citep{EH-SPN-GW:93}. Contraction theory has recently found applications in neural networks~\citep{MR-IM:20}.

\paragraph*{Contributions}

Our contributions are as follows:
\begin{itemize}
\item We establish how the contraction analysis of a dynamical system immediately characterizes the stability of its associated implicit and explicit Euler integrations. We consider the case where the dynamical system is contracting and time-varying without necessarily possessing the same equilibrium point at all times; the time-invariant case follows immediately. In the case of the implicit Euler method, compared to the work~\citep{CAD-HH:72}, our results: 1) do not assume the vector field to be continuously differentiable; 2) consider the case where the contraction rate can be time-varying; and 3) use the one-sided Lipschitz condition. 
In the case of the explicit Euler method, our results generalize the classic forward step method from the theory of monotone operators~\citep{EKR-SB:16} because we consider the case of time-varying contraction rates.
\end{itemize}
For the rest of contributions, we remark that we consider a differentiable, strongly convex objective function with Lipschitz smoothness 
--- a commonly used class of functions 
for the theoretical analysis of 
optimization flows 
and their discretization schemes, e.g.,~\citep{WS-SB-EJC:16,ACW-BR-MIJ:18,BS-SSD-MIJ-WJS:18,MV-JC:19,BS-SSD-WS-MIJ:19,MM-MJ:19}, and even for other state-of-the-art solvers for time-varying optimization, e.g.,~\citep{NB-AS-RC:20}. 
\begin{itemize}
\item We formulate a system of ODEs called the \emph{ACcelerated-COntracting-NESTerov} (ACCONEST) flow, and we prove it is an optimization flow with exponential convergence using contraction theory, from which we directly prove that both its implicit and explicit Euler integrations have linear convergence. 
For the implicit integration, rate-matching and acceleration is guaranteed. For the appropriate integration step, the explicit Euler integration of this system is the Nesterov acceleration method and thus has rate-matching and acceleration too.
\item Our contraction analysis for the ACCONEST flow implies the
  existence of a simple quadratic Lyapunov function for the optimization
  flow, as opposed to other more complex Lyapunov functions or energy
  functionals used in previous works, e.g.,
  see~\citep{WS-SB-EJC:16,BS-SSD-WS-MIJ:19}. Proving that the
    optimization flow is contracting provides additional robustness results
    that a Lyapunov analysis does not generally imply, such as strong
    input-to-state stability guarantees and finite input-state
    gains~\citep{AD-SJ-FB:20o}. Moreover, a contracting system is guaranteed
    to have fast correction after transient perturbations to the trajectory
    of the solution, since initial conditions are forgotten.
\item Finally, we extend our analysis to the case of time-varying optimization, i.e., where the objective function varies through time. 
Contraction theory provides guarantees for the tracking of the time-varying optimization solution based upon the analysis done for the time-invariant case. We show that the ACCONEST flow and its discretizations have a tracking error that is uniformly ultimately bounded.
\end{itemize}

\paragraph*{Paper organization} 
Section~\ref{sec:prelim} has notation and preliminary concepts. Section~\ref{sec:contr-th} has results on contraction theory and stability of discretizations. Section~\ref{sec:opt-fl} uses contraction theory to analyze a proposed optimization flow and its associated optimization algorithms. 
Section~\ref{sec:tv-opt} analyzes the time-varying case. Section~\ref{sec:concl} is the conclusion.

The proofs for the results in Section~\ref{sec:contr-th}, Section~\ref{sec:opt-fl}, and Section~\ref{sec:tv-opt} are found in the Appendix.

\section{Preliminaries and notation}
\label{sec:prelim}

\subsection{Notation and definitions}
Consider $A\in\R^{n\times{n}}$. 
If $A$ only has 
real eigenvalues, let $\lambda_{\min}(A)$ be its minimum eigenvalue and $\lambda_{\max}(A)$ its 
maximum one. Let $I_n$ be the $n\times n$ identity matrix. 
Let $\norm{\cdot}_p$ denote the $\ell_p$-norm, and when the argument of a norm is a matrix, we refer to its respective induced norm. 
The matrix measure associated 
to norm $\norm{\cdot}$ is $\mu(A) = \lim_{h\to 0^{+}}\frac{\|I_n+hA\|-1}{h}$. 
Given invertible $Q\in\R^{n\times n}$, let $\norm{\cdot}_{2,Q}$ be the weighted $\ell_2$-norm $\norm{x}_{2,Q}=\norm{Qx}_2$
, $x\in\R^n$, and whose associated matrix measure is 
$\mu_{2,Q}(A)=\mu_{2}(QAQ^{-1})$.
  
Let $\ones_n$ and
$\zeros_n$ be the all-ones and all-zeros column vector with $n$
entries respectively. 
Given $x_i\in\R^{k_i}$, 
let $(x_1,\dots,x_N)=\begin{bmatrix}x_1^\top &\dots&x_N^\top\end{bmatrix}$.

Consider a continuously differentiable function $f:\R^n\to\R$. We say $f$ is \emph{$L$-smooth} if $\norm{\nabla f(x)-\nabla f(y)}_2\leq L\norm{x-y}_2$, $L>0$, for any $x,y\in\R^n$; and \emph{$\mu$-strongly convex} if $\mu\norm{x-y}_2^2\leq (\nabla f(x)-\nabla f(y))^\top(x-y)$, $\mu>0$, for any $x,y\in\R^n$. Let $\SuL$ be the set of continuously differentiable functions on $\R^n$ that are $L$-smooth and $\mu$-strongly convex, and denote by $\kappa:=\frac{L}{\mu}$ the condition number of any function belonging to this set.

Consider a dynamical system $\dot{x}=g(x,t)$ with \emph{vector field} $g:\R^n\times [t_0,\infty)\to \R^n$, $x\in\R^n$, $t\geq t_0$. We assume that a solution exists under any initial condition. Let $t\mapsto\phi(t,t_0,x_0)$ be the trajectory of
  the system 
  starting from $x_0\in\R^n$ at time $t_0\geq 0$. We say $g$ is \emph{uniformly $\ell$-Lipschitz continuous} with respect to norm $\norm{\cdot}$ if $\norm{g(x,t)-g(y,t)}\leq \ell\norm{x-y}$, $\ell>0$, for any $x,y\in\R^n$, and for every $t>0$. 
 The \emph{implicit Euler} discretization or method with constant discretization or integration step-size $h>0$ of this dynamical system with $t_0=0$ is: $y_{k+1}=y_k+hg(y_{k+1},(k+1)h)$, $k=0,1,\dots$, with given $y_0\in\R^n$. The \emph{explicit Euler} discretization or method with constant discretization or integration step-size $h>0$ of this dynamical system with $t_0=0$ is: $y_{k+1}=y_k+hg(y_{k},kh)$, $k=0,1,\dots$, with given $y_0\in\R^n$.

\subsection{Review of basic results on contraction theory}

The following results 
will be useful throughout the paper; e.g., see references~\citep{GEL-VL:72,AD-SJ-FB:20o}.

\begin{theorem}[Characterizations of contracting systems]
\label{th:contr}
Consider the dynamical system $\dot{x}=g(x,t)$ with $x\in\R^n$. Pick a symmetric positive-definite $P\in\R^{n\times{n}}$ and an integrable function $\bar{\gamma}:[t_0,\infty)\to\R$. Then, the following statements are equivalent
\begin{enumerate}
\item\label{it:contr-1} $\norm{\phi(t,t_0,x_0)-\phi(t,t_0,y_0)}_{2,P^{1/2}}\leq e^{-\int_{t_0}^t \bar{\gamma}(s)ds}\norm{x_0-y_0}_{2,P^{1/2}}$;
\item\label{it:contr-2} the \emph{one-sided Lipschitz condition} $(y-x)^\top P(g(x,t)-g(y,t))\leq -\bar{\gamma}(t)\norm{x-y}_{2,P^{1/2}}^2$ holds for every $x,y\in\R^n$ and $t\geq t_0$.
\setcounter{saveenum}{\value{enumi}}
\end{enumerate} 
If $f$ is continuously differentiable, then the previous statements are also equivalent to
\begin{enumerate}\setcounter{enumi}{\value{saveenum}}
\item\label{it:contr-3} $\mu_{2,P^{1/2}}\left(Dg(x,t)\right)\leq -\bar{\gamma}(t)$ for every $x\in\R^n$ and $t\geq t_0$; where $D(\cdot,t)$ is the Jacobian of $g(\cdot,t)$.
\end{enumerate}
\end{theorem}

Any dynamical system that satisfies any of the conditions in Theorem~\ref{th:contr} with $\int_{t_0}^t \bar{\gamma}(s)ds>0$ for any $t\geq t_0$ is said to be contracting (with respect to~$\norm{\cdot}_{2,P^{1/2}}$) or have the \emph{contraction} property, and $\bar{\gamma}(t)$ is the contraction rate. 
The contraction rate is time-invariant when $\bar{\gamma}(t)\equiv \gamma$.

The central and most challenging part of determining if a system is contracting is to find an adequate symmetric positive-definite matrix $P\in\R^{n\times{n}}$ that establishes its contraction 
respect to $\norm{\cdot}_{2,P^{1/2}}$.

\begin{theorem}[Characterization for homogeneous systems]
\label{th:contr2}
Consider the homogeneous or time-invariant dynamical system $\dot{x}=g(x)$ with $x\in\R^n$. Pick a symmetric positive-definite $P\in\R^{n\times{n}}$. If the system is contracting with contraction rate $\gamma>0$ and respect to $\norm{\cdot}_{2,P^{1/2}}$, then
\begin{enumerate}
\item there exists a unique equilibrium point $x^*$;
\item \label{th:contr2-lab1} $x^*$ is exponentially globally stable with Lyapunov functions $V(x)=\norm{x-x^*}_{2,P^{1/2}}^2$ and $V(x)=\norm{g(x)}_{2,P^{1/2}}^2$; and 
\item \label{th:contr2-lab2} the exponential convergence rate is $\gamma$.
\end{enumerate}
\end{theorem}

In light of Theorem~\ref{th:contr}, statements~\ref{th:contr2-lab1} and~\ref{th:contr2-lab2} of Theorem~\ref{th:contr2} can be equivalently expressed as $\norm{\phi(t,t_0,x_0)-x^*}\leq e^{-\gamma(t-t_0)}\norm{x_0-x^*}$ for any $x_0\in\R^n$.

\section{Contraction Theory and stability of Euler discretizations}
\label{sec:contr-th}

The next result establishes how the contraction analysis of a dynamical system can immediately characterize the stability of its associated implicit and implicit Euler integration.

\begin{theorem}[Stability of the implicit Euler integration]
\label{th:dis-contr}
Consider the system $\dot{x}=g(x,t)$, $x\in\R^n$, $t\geq 0$, and that it is contracting with contraction rate $\bar{\gamma}(t)>0$ for all $t\geq 0$ and respect to some appropriate norm $\norm{\cdot}_{2,P^{1/2}}$, $P\in\R^{n\times n}$. 
Let $(y_k)$ be a sequence generated by the implicit Euler discretization with constant integration step-size $h>0$. 
\begin{enumerate}
\item\label{it:dis-1} If there exists 
$x^*$ such that $g(x^*,t)=\zeros_n$ for all $t\geq 0$, then
\begin{multline}
\label{eq:thcontr-1}
\norm{y_k-x^*}_{2,P^{1/2}}\\\leq \left(\prod^k_{m=1}(1+h\bar{\gamma}(mh))^{-1}\right)\norm{y_0-x^*}_{2,P^{1/2}},
\end{multline} 
for $k=1,2,\dots$; and if additionally the contraction rate is time-invariant $\bar{\gamma}(t)\equiv\gamma>0$, then there is linear convergence
\begin{equation}
\norm{y_k-x^*}_{2,P^{1/2}}\leq (1+h\gamma)^{-k}\norm{y_0-x^*}_{2,P^{1/2}}.
\end{equation} 
\item\label{it:dis-2} If there exists a curve $t\mapsto x^*(t)$ such that $g(x^*(t),t)=\zeros_n$ for any $t\geq 0$, then
\begin{equation}
\label{eq:thcontr-2}
\begin{aligned}
&\norm{y_k-x^*(kh)}_{2,P^{1/2}}\\&\leq \left(\prod^k_{m=1}(1+h\bar{\gamma}(mh))^{-1}\right)\norm{y_0-x^*(0)}_{2,P^{1/2}}\\
&\quad+\sum^k_{m=1}\prod^k_{r=m}(1+h\bar{\gamma}(rh))^{-1}\\
&\quad\quad\times\norm{x^*(mh)-x^*((m-1)h)}_{2,P^{1/2}},
\end{aligned}
\end{equation}
for $k=1,2,\dots$; and if additionally the contraction rate is time-invariant $\bar{\gamma}(t)\equiv\gamma>0$ and 
$\sup_{k=1,2,\dots}\norm{x^*(kh)-x^*((k-1)h)}_{2,P^{1/2}}\leq\rho$ for some constant $\rho\geq 0$, then
\begin{multline}
\label{eq:thcontr-3}
\norm{y_k-x^*(kh)}_{2,P^{1/2}}\\\leq (1+h\gamma)^{-k}\norm{y_0-x^*(0)}_{2,P^{1/2}}+\rho\sum^k_{m=1}(1+h\gamma)^{-m},
\end{multline} 
and so
\begin{equation}
\label{eq:thcontr-4}
\limsup_{k\to\infty}\norm{y_k-x^*(kh)}_{2,P^{1/2}}\leq \frac{\rho}{h\gamma}.
\end{equation} 
\end{enumerate} 
\end{theorem}
 
\begin{remark}[Comparison with the work~\citep{CAD-HH:72}]
  Statement~\ref{it:dis-1} in Theorem~\ref{th:dis-contr} was proved
  in~\citep{CAD-HH:72} for the case where the vector field of the dynamical system is continuously differentiable and using the condition~\ref{it:contr-3} of Theorem~\ref{th:contr} along with properties of matrix measures. We also remark that Theorem~\ref{th:dis-contr} considers the case where the contraction rate can be time-varying, unlike~\citep{CAD-HH:72}.
\end{remark}

\begin{remark}[Rate-matching for the Implicit Euler integration]
\label{rem:imE}
Assume the conditions of statement~\ref{it:dis-1} of Theorem~\ref{th:dis-contr}, with the system 
having a time-invariant contraction rate $\gamma>0$. 
Now, observe that 
$e^{\gamma h}\geq (1+\gamma h)\implies (1+\gamma h)^{-k}\leq e^{-\gamma hk}$ for $k=1,2\dots$. In order words, the linear convergence of the implicit Euler integration is upper bounded by the exponential convergence of its ODE counterpart for the discretized time $t=kh$ (see statement~\ref{it:contr-1} of Theorem~\ref{th:contr}), i.e., there is rate-matching. 
\end{remark}

\begin{theorem}[Stability of the explicit Euler integration]
\label{th:dis-contr2}
Consider the system $\dot{x}=g(x,t)$, $x\in\R^n$, $t\geq 0$, with $g$ being uniformly $\ell$-Lipschitz continuous, and that it is contracting with contraction rate $\bar{\gamma}(t)$ such that $\inf_{t\geq 0}\bar{\gamma}(t)>0$ and $\ell>\sup_{t\geq 0}\bar{\gamma}(t)$, 
both respect to some appropriate norm $\norm{\cdot}_{2,P^{1/2}}$, $P\in\R^{n\times n}$. 
Let $(y_k)$ be a sequence generated by the explicit Euler discretization with constant integration step-size 
$0<h<\frac{2}{\ell^2}\inf_{t\geq 0}\bar{\gamma}(t)$. 
\begin{enumerate}
\item\label{it:dis-12} If there exists $x^*$ such that $g(x^*,t)=\zeros_n$ for all $t\geq 0$, then
\begin{multline}
\label{eq:thcontr-12}
\norm{y_k-x^*}_{2,P^{1/2}}\\\leq \left(\prod^{k-1}_{m=0}(1-2h\bar{\gamma}(mh)+h^2\ell^2)^{1/2}\right)\\\times\norm{y_0-x^*}_{2,P^{1/2}},
\end{multline} 
for $k=1,2,\dots$; and if additionally the contraction rate is time-invariant $\bar{\gamma}(t)\equiv\gamma>0$, then there is linear convergence
\begin{equation}
\label{eq:thcontr-1a2}
\norm{y_k-x^*}_{2,P^{1/2}}\leq (1-2h\gamma+h^2\ell^2)^{k/2}\norm{y_0-x^*}_{2,P^{1/2}}.
\end{equation}
\item\label{it:dis-22} If there exists a curve $t\mapsto x^*(t)$ such that $g(x^*(t),t)=\zeros_n$ for any $t\geq 0$, then
\begin{multline}
\label{eq:thcontr-22}
\norm{y_k-x^*(kh)}_{2,P^{1/2}}\\\leq \left(\prod^{k-1}_{m=0}(1-2h\bar{\gamma}(mh)+h^2\ell^2)^{1/2}\right)\\\times\norm{y_0-x^*(0)}_{2,P^{1/2}}\\
+1_{\{k>1\}}\sum^{k-1}_{m=1}\prod^{k-1}_{r=m}(1-2h\bar{\gamma}(rh)+h^2\ell^2)^{1/2}\\
\times\norm{x^*(mh)-x^*((m-1)h)}_{2,P^{1/2}}\\
+\norm{x^*(kh)-x^*((k-1)h)}_{2,P^{1/2}},
\end{multline}
for $k=1,2,\dots$; where $1_{\{k>1\}}=1$ if $k>1$ and $1_{\{k>1\}}=0$ otherwise. Additionally, if 
the contraction rate is time-invariant $\bar{\gamma}(t)\equiv\gamma>0$ and 
$\sup_{k=1,2,\dots}\norm{x^*(kh)-x^*((k-1)h)}_{2,P^{1/2}}\leq\rho$ for some constant $\rho\geq0$, then
\begin{multline}
\label{eq:thcontr-32}
\norm{y_k-x^*(kh)}_{2,P^{1/2}}\\\leq (1-2h\gamma+h^2\ell^2)^{k/2}\norm{y_0-x^*(0)}_{2,P^{1/2}}\\
+\rho\sum^{k-1}_{m=0}(1-2h\gamma+h^2\ell^2)^{m/2},
\end{multline}
for $k=1,2,\dots$, and so
\begin{multline}
\label{eq:thcontr-42}
\limsup_{k\to\infty}\norm{y_k-x^*(kh)}_{2,P^{1/2}}\\\leq \frac{\rho}{1-(1-2h\gamma+h^2\ell^2)^{1/2}}.
\end{multline} 
\item\label{it:dis-32} In all the previous cases where the system has contraction rate $\gamma$, the optimal choice for the time-step that gives the fastest convergence rate is 
$h^*=\frac{\gamma}{\ell^2}$; i.e., in equations~\eqref{eq:thcontr-1a2},~\eqref{eq:thcontr-32} and~\eqref{eq:thcontr-42}, we have $1-2h^*\gamma+{h^*}^2\ell^2= 1-\frac{\gamma^2}{\ell^2}$.
\end{enumerate} 
\end{theorem}
\begin{remark}[Connection with monotone operator theory]
Consider the homogeneous dynamical system $\dot{x}=g(x)$, $x\in\R^n$, with time-invariant contraction rate $\gamma>0$. If the vector field $g$ is seen as an operator (over its argument), then the one-sided Lipschitz condition in item~\ref{it:contr-2} of Theorem~\ref{th:contr} is equivalent to $-g$ being \emph{strongly monotone} with parameter $\gamma$. 
Therefore, the classic proof for finding zeros of a strongly monotone operator~\citep{EKR-SB:16} (also called the forward step method) follows as a particular case of the proof of item~\ref{it:dis-12} of Theorem~\ref{th:dis-contr2} for the case of time-invariant contraction rates.
\end{remark}

\begin{remark}[Improving the contraction rate for Theorem~\ref{th:dis-contr2}]
\label{rem:contr-Eul-improv}
Consider Theorem~\ref{th:contr} and its assumptions, with the system having a time-invariant contraction rate $\gamma>0$. A closer look to the proof shows that the linear convergence rate depends on the inequality 
\begin{multline}
\label{eq:express-up}
-2h\gamma\norm{y_k-x^*}_{2,P^{1/2}}^2+h^2\norm{g(y_k,kh)}_{2,P^{1/2}}^2\\\leq (-2h\gamma+h^2\ell^2)\norm{y_k-x^*}_{2,P^{1/2}}^2.
\end{multline}
The use of the Lipschitz constant $\ell$ is done in order to provide a general bound for our proof; however, it may be possible to upper bound the left-hand side of~\eqref{eq:express-up} and obtain a less 
conservative result, and thus a better contraction rate. However, this bound will depend on the particular vector field $g$. 
\end{remark}

\section{The Accelerated-Contracting-Nesterov flow and its Euler discretizations}
\label{sec:opt-fl}

Consider an objective function $f:\R^n\to\R$ satisfying $f\in\SuL$. 
The works~\citep{WS-SB-EJC:16,BS-SSD-MIJ-WJS:18} propose the following system of ODEs: 
\begin{equation}
\label{eq:opt-ODE}
\begin{aligned}
\dot{x}_1&=c\,x_2\\
\dot{x}_2&=-a\,x_2-b\nabla f(x_1)
\end{aligned}
\end{equation}
with $x_1,x_2\in\R^n$, 
as an optimization flow related to the heavy-ball method for specific values of the scalars $a,b,c>0$. Then, the work~\citep{MM-MJ:19} proposed to modify~\eqref{eq:opt-ODE} by including a \emph{displaced gradient}:
\begin{equation}
\label{eq:syst-dispgr}
\begin{aligned}
\dot{x}_1&=c\,x_2\\
\dot{x}_2&=-a\,x_2-b\nabla f(x_1+dx_2)
\end{aligned}
\end{equation}
with additional scalar $d>0$. 
The state $x_1$ is the \emph{position} of the system whose trajectory should converge to the minimum of $f$, and $x_2$ is its instantaneous velocity. In this paper, for any fixed $t\geq 0$, we interpret $x_1(t)+d x_2(t)$ as the \emph{predicted} position that the system has if it constantly follows the velocity $x_2(t)$ for a period of time of length $d$. Thus, we can think of system~\eqref{eq:syst-dispgr} as having its acceleration $\dot{x}_2$ being influenced by the gradient at a predicted position, instead of its current position as in~\eqref{eq:opt-ODE}.

Now, recall that in a gradient flow the negative gradient affects the direction of the instantaneous velocity, i.e., $\dot{x}_1=-\nabla f(x_1)$. Then, we propose to include such effect of the gradient in~\eqref{eq:syst-dispgr} 
with the expectation that 
this may improve the convergence rate of the system:  
\begin{equation}
\label{eq:syst-funny}
\begin{aligned}
\dot{x}_1&=c\,x_2-e\nabla f(x_1+dx_2)\\
\dot{x}_2&=-a\,x_2-b\nabla f(x_1+dx_2)
\end{aligned}
\end{equation}
with additional scalar $e>0$. Systems~\eqref{eq:opt-ODE},~\eqref{eq:syst-dispgr}, and~\eqref{eq:syst-funny} have the unique equilibrium point $(x^*,\zeros_n)$ with $x^*=\arg\min_{z\in\R^n} f(z)$.

Now, we make the change of variables $\bar{x}_1=x_1$ and $\bar{x}_2=x_1+dx_2$, and 
set $a=\frac{2}{\sqrt{\kappa}+1}$, $b=\frac{\sqrt{\kappa}-1}{2L}$, $c=d=\frac{\sqrt{\kappa}-1}{\sqrt{\kappa}+1}$ and $e=\frac{1}{L}$. Then, we obtain the \emph{ACcelerated-COntracting-NESTerov} (ACCONEST) flow 
\begin{equation}
\label{eq:syst-funny2}
\begin{aligned}
\begin{bmatrix}
\dot{\bar{x}}_1\\
\dot{\bar{x}}_2
\end{bmatrix}
:=&\fAcc(\bar{x}_1,\bar{x}_2)\\
=&
\begin{bmatrix}
\bar{x}_2-\bar{x}_1-\frac{1}{L}\nabla f(\bar{x}_2)\\
\frac{\sqrt{\kappa}-1}{\sqrt{\kappa}+1}(\bar{x}_2-\bar{x}_1)-\frac{2\sqrt{\kappa}}{(\sqrt{\kappa}+1)L}\nabla f(\bar{x}_2)
\end{bmatrix},
\end{aligned}
\end{equation}
which has unique equilibrium point $(x^*,x^*)$; i.e., 
at equilibrium, the system's current position $\bar{x}_1$ and predicted position $\bar{x}_2$ must be equal to the optimal value. 

We now analyze the ACCONEST flow 
using contraction theory and the results introduced in the previous section.

\begin{theorem}[Analysis of the ACCONEST flow 
and its optimization algorithms]
\label{th:main-acc}
Consider the ACCONEST flow
~\eqref{eq:syst-funny2} with $f:\R^n\to\R$, $f\in\SuL$; and let the optimal value $x^*=\arg\min_{z\in\R^n}f(z)$.
  \begin{enumerate}
  \item\label{1-1} The ACCONEST flow is contracting with rate $\sqrt{\frac{\mu}{L}}$ and respect to $\norm{\cdot}_{2,P^{1/2}}$ with 
$P=\begin{bmatrix}
  \gamma\frac{\sqrt{\kappa}}{\sqrt{\kappa}+1}&-1\\
  -1&\frac{\sqrt{\kappa}+1}{\sqrt{\kappa}}  
  \end{bmatrix}\otimes I_n$, 
  $1<\gamma\leq 1+\frac{1}{\kappa}$.
  \item \label{1-2} The ACCONEST flow has global exponential convergence to $\begin{bmatrix}1\\1\end{bmatrix}\otimes x^*$.
  \setcounter{saveenum}{\value{enumi}}
  \end{enumerate}
\begin{enumerate}\setcounter{enumi}{\value{saveenum}}
\item\label{1-3} 
Let $((y_k^{(1)},y^{(2)}_k))$ be the sequence generated by the implicit Euler discretization of the ACCONEST flow~\eqref{eq:syst-funny2} with integration step-size $h$. Then, there is 
linear convergence to the optimal value characterized by
	\begin{equation}
	\label{eq:conv-1iEd}
	f(y^{(1)}_k)-f(x^*)\leq C\,\left(1+h\sqrt{\frac{\mu}{L}}\right)^{-2k}
	\end{equation}  
  with $C$ being some constant that depends on the initial conditions $y^{(1)}_0,y^{(2)}_0\in\R^n$.
\item\label{1-4}
Let $((y_k^{(1)},y^{(2)}_k))$ be the sequence generated by the explicit Euler discretization of the ACCONEST flow~\eqref{eq:syst-funny2} with integration step-size $h^*:=\frac{1}{\sqrt{\kappa}(7+5\beta+6\beta^2)}\frac{\lambda_{\min}(P)}{\lambda_{\max}(P)}$, $\beta:=\frac{\sqrt{\kappa}-1}{\sqrt{\kappa}+1}$. Then, there is 
%
%
linear convergence to the optimal value characterized by
	\begin{equation}
	\label{eq:conv-1iEd}
	f(y^{(1)}_k)-f(x^*)\leq C\,\left(1-h^*\sqrt{\frac{\mu}{L}}\right)^{k}
	\end{equation}  
  with $C$ being some constant that depends on the initial conditions $y^{(1)}_0,y^{(2)}_0\in\R^n$.
  \setcounter{saveenum}{\value{enumi}}
\end{enumerate}
\end{theorem}

\begin{remark}[About Theorem~\ref{th:main-acc}]
Some remarks are in order:

\begin{itemize}
\item\textbf{Rate-matching and acceleration:} 
Note that, as pointed out in Remark~\ref{rem:imE}, we have rate-matching for the implicit Euler integration (starting with step-size equal to one) 
and, consequently, there is acceleration because the contraction rate of the optimization flow~\eqref{eq:syst-funny2} is $\sqrt{\frac{\mu}{L}}$.
Unfortunately, for the explicit Euler case, we could not prove rate-matching in Theorem~\ref{th:main-acc}; however, the simple proof of Proposition~\ref{prop:acc} shows that rate-matching and acceleration can be established (since $(1-c)^{-k}\leq e^{-ck}$ for any $0<c<1$ and non-negative integer $k$).
\item\textbf{Connection with Lyapunov analysis:} As a consequence of Theorem~\ref{th:contr2}, the optimization flow~\eqref{eq:syst-funny2} can have its exponential convergence certified by two simple quadratic Lyapunov functions as stated in Corollary~\ref{cor-Lyap}. This is in contrast to other works that have analyzed optimization flows using more complicated Lyapunov functions, e.g.,~\citep{WS-SB-EJC:16,BS-SSD-MIJ-WJS:18,MV-JC:19,BS-SSD-WS-MIJ:19}.
\end{itemize}
\end{remark}

\begin{corollary}[Quadratic Lyapunov functions for the optimization flow~\eqref{eq:syst-funny2}]
\label{cor-Lyap}
Consider the ACCONEST flow~\eqref{eq:syst-funny2} with $f:\R^n\to\R$, $f\in\SuL$. Then, the system is an optimization flow 
whose exponential convergence to the optimization solution with rate $\sqrt{\frac{\mu}{L}}$ can be established via the Lyapunov functions $V(\bar{x}_1,\bar{x}_2)=\norm{(\bar{x}_1-x^*,\bar{x}_2-x^*)^\top}_{2,P^{1/2}}^2$ and $V(\bar{x}_1,\bar{x}_2)=\norm{\fAcc(\bar{x}_1,\bar{x}_2)}_{2,P^{1/2}}^2$ with 
$P=\begin{bmatrix}
  \gamma\frac{\sqrt{\kappa}}{\sqrt{\kappa}+1}&-1\\
  -1&\frac{\sqrt{\kappa}+1}{\sqrt{\kappa}}  
  \end{bmatrix}\otimes I_n$, 
  $1<\gamma\leq 1+\frac{1}{\kappa}$.
\end{corollary}

\begin{proposition}[Acceleration from the explicit Euler integration of the ACCONEST flow~\eqref{eq:syst-funny2}]
\label{prop:acc}
Consider the ACCONEST flow~\eqref{eq:syst-funny2} with $f:\R^n\to\R$, $f\in\SuL$.
Let $((y_k^{(1)},y^{(2)}_k))$ be the sequence generated by its explicit Euler discretization with integration step-size $1$. Then, there is 
%
%
linear convergence to the optimal value characterized by
	\begin{equation}
	\label{eq:conv-1iEd1}
	f(y^{(1)}_k)-f(x^*)\leq C\,\left(1-\sqrt{\frac{\mu}{L}}\right)^{k}
	\end{equation}  
  with $C$ being some constant that depends on the initial conditions $y^{(1)}_0,y^{(2)}_0\in\R^n$.
\end{proposition}
\begin{proof}
We compute the update for the explicit Euler integration of the optimization flow~\eqref{eq:syst-funny2} with integration step-size $1$ which, after some algebraic operations, becomes 
\begin{equation*}
\begin{aligned}
y^{(1)}_{k+1}&=y^{(2)}_{k}-\frac{1}{L}\nabla f(y^{(2)}_{k})\\
y^{(2)}_{k+1}&=y^{(1)}_{k+1}+\frac{\sqrt{L}-\sqrt{\mu}}{\sqrt{L}+\sqrt{\mu}}(y^{(1)}_{k+1}-y^{(1)}_{k})
\end{aligned}
\end{equation*}
which is the Nesterov acceleration method itself~\citep{YN:18}. Thus, the linear convergence~	\eqref{eq:conv-1iEd1} is guaranteed.
\end{proof}

The previous proposition tells us that if we perform an explicit Euler integration to the optimization flow~\eqref{eq:syst-funny2} with integration step-size $1$, then we obtain the Nesterov acceleration method itself. Thus, we have used contraction theory in Theorem~\ref{th:main-acc} to analyze an optimization flow associated to the Nesterov accelerated method (hence the name of Accelerated-Contracting-Nesterov flow). Moreover, this optimization flow has its exponential rate being identical to Nesterov's acceleration rate, which makes it quite different in structure and convergence rate from other flows in the literature (e.g., the ones of the form of~\eqref{eq:opt-ODE}), which have not been proved to be contracting.

\begin{remark}[Comparing the ACCONEST flow to high-resolution ODEs]
For the class of objective functions in $\SuL$, the works~\citep{WS-SB-EJC:16,BS-SSD-MIJ-WJS:18}, and other consecutive ones that are based on the use of their equations such as~\citep{MV-JC:19,BS-SSD-WS-MIJ:19}, analyze the so-called high-resolution ODEs, which have some specific fixed values for the constants of the system~\eqref{eq:opt-ODE}. 
We also remark that the high-resolution ODE associated to the Nesterov acceleration method considers the Hessian $\nabla^2 f(x_1)$ in its structure – whereas the ACCONEST flow~\eqref{eq:syst-funny2}, which also corresponds to the Nesterov acceleration method, makes no use of second order information, let alone any twice-differentiability assumption. Lyapunov analysis of this high-resolution ODE and its discretizations has shown that acceleration does not occur under the explicit Euler integration~\citep{BS-SSD-WS-MIJ:19}. 
\end{remark}

\section{Solving the time-varying optimization with the Accelerated-Contracting-Nesterov flow and its Euler discretizations}
\label{sec:tv-opt}
Assume that the optimization problem is time-varying
\begin{equation}\label{eq:opt-onl}
	\min_{z\in\R^n} f(z,t)
\end{equation}
and that, for every $t\geq 0$, $f(\cdot,t)\in\SuL$. 
Similar to~\eqref{eq:syst-funny2}, we consider the following system which we call the time-varying ACCONEST flow:
\begin{equation}
\label{eq:opt-ODE-tv}
\begin{bmatrix}
\dot{\bar{x}}_1\\
\dot{\bar{x}}_2
\end{bmatrix}=
\begin{bmatrix}
\bar{x}_2-\bar{x}_1-\frac{1}{L}\nabla f(\bar{x}_2,t)\\
\frac{\sqrt{\kappa}-1}{\sqrt{\kappa}+1}(\bar{x}_2-\bar{x}_1)-\frac{2\sqrt{\kappa}}{(\sqrt{\kappa}+1)L}\nabla f(\bar{x}_2,t)
\end{bmatrix}
\end{equation}
with $\bar{x}_1,\bar{x}_2\in\R^n$. 

Given a fixed time $t\geq 0$, let $x^*(t)=\arg\min_{x\in\R^n}f(x,t)$. 
Then, $(x^*(t))_{t\geq 0}$ defines the \emph{optimizer trajectory}, i.e., the trajectory of the solution to the optimization problem through time. 
The following result establishes the performance of the time-varying ACCONEST flow~\eqref{eq:opt-ODE-tv} and its discretizations in tracking the optimizer trajectory. 

\begin{theorem}[Performance of the time-varying ACCONEST flow and its optimization algorithms for time-varying optimization]
\label{th:tv}
Consider the time-varying ACCONEST flow~\eqref{eq:opt-ODE-tv} with $f:\R^n\times\R\to\R$, $f(\cdot,t)\in\SuL$ for every $t\geq 0$. Set $z(t):=(x_1(t),x_2(t))^\top$ and $z^*(t):=(x^*(t),x^*(t))^\top$.
\begin{enumerate}
\item\label{tv-1} The time-varying ACCONEST flow is contracting with rate $\sqrt{\frac{\mu}{L}}$ and respect to $\norm{\cdot}_{2,P^{1/2}}$ with $P=\begin{bmatrix}
  \gamma\frac{\sqrt{\kappa}}{\sqrt{\kappa}+1}&1\\
  1&\frac{\sqrt{\kappa}+1}{\sqrt{\kappa}}  
  \end{bmatrix}\otimes I_n$, 
  $1<\gamma\leq 1+\frac{1}{\kappa}$. 
\item\label{tv-2} Consider that $x\mapsto f(t,x)$ is twice continuously differentiable for any $t\geq 0$ and that $t\mapsto \nabla f(t,x)$ is continuously differentiable with $\norm{\dot{\nabla}f(x,t)}_2\leq \rho$, $\rho>0$, for any $(t,x)\in[0,\infty)\times\R^n$. Then,
\begin{multline}
\label{tracking-pd}
\norm{z(t)-z^*(t)}_{2,P^{1/2}}\\
\leq \left(\norm{z(0)-z^*(0)}_{2,P^{1/2}}-\frac{\sqrt{2\lambda_{\max}(P)}\rho}{\mu}\sqrt{\frac{L}{\mu}}\right)\\
\times e^{-\sqrt{\frac{\mu}{L}}t}+\frac{\sqrt{2\lambda_{\max}(P)}\rho}{\mu}\sqrt{\frac{L}{\mu}},\;\; t\geq 0,
\end{multline}
and so the tracking error is uniformly ultimately bounded as 
\begin{equation}
\label{eq:asympt-cont}
\limsup_{t\to\infty}\norm{z(t)-z^*(t)}_{2,P^{1/2}}\leq \frac{\sqrt{2\lambda_{\max}(P)}\rho}{\mu}\sqrt{\frac{L}{\mu}}.
\end{equation}
  \setcounter{saveenum}{\value{enumi}}
\end{enumerate}
Now, consider that $\sup_{k=1,2,\dots}\norm{\nabla f(x,kh)-\nabla f(x,(k-1)h)}_2\leq \rho$, $\rho>0$, for any $x\in\R^n$.
\begin{enumerate}\setcounter{enumi}{\value{saveenum}}
\item\label{tv-3} Let $(y_k:=(y_k^{(1)},y^{(2)}_k))$ be the sequence generated by the implicit Euler discretization of the time-varying ACCONEST flow~\eqref{eq:opt-ODE-tv} with integration step-size $h$. Then, 
\begin{multline}
\label{eq:thcontr-3-tv}
\norm{y_k-z^*(kh)}_{2,P^{1/2}}\\\leq (1+h\sqrt{\frac{\mu}{L}})^{-k}\norm{y_0-z^*}_{2,P^{1/2}}\\
+\sqrt{2\lambda_{\max}(P)}\frac{\rho}{\mu}\sum^k_{m=1}(1+h\sqrt{\frac{\mu}{L}})^{-m} 
\end{multline} 
for $k=1,2,\dots$; and so the tracking error is uniformly ultimately bounded as 
\begin{equation}
\label{eq:asympt-dis}
\limsup_{k\to\infty}\norm{y_k-z^*(kh)}_{2,P^{1/2}}\leq \frac{\sqrt{2\lambda_{\max}(P)}\rho}{h\mu}\sqrt{\frac{L}{\mu}}.
\end{equation} 
\item\label{tv-4} Let $(y_k:=(y_k^{(1)},y^{(2)}_k))$ be the sequence generated by the explicit Euler discretization of the time-varying ACCONEST flow~\eqref{eq:opt-ODE-tv} with integration step-size $h^*$ as in statement~\ref{1-4} of Theorem~\ref{th:main-acc}. Then,
\begin{multline}
\label{eq:thcontr-4-tv}
%
\norm{y_k-x^*(kh)}_{2,P^{1/2}}\\
\leq \left(1-h^*\sqrt{\frac{\mu}{L}}\right)^{k/2}\norm{y_0-x^*(0)}_{2,P^{1/2}}\\
+\sqrt{2\lambda_{\max}(P)}\frac{\rho}{\mu}\sum^{k-1}_{m=0}\left(1-h^*\sqrt{\frac{\mu}{L}}\right)^{m/2},
\end{multline}
for $k=1,2,\dots$; and so the tracking error is uniformly ultimately bounded as 
\begin{equation}
\label{eq:asympt-dis2}
\limsup_{k\to\infty}\norm{y_k-x^*(kh)}_{2,P^{1/2}}\leq \frac{\sqrt{2\lambda_{\max}(P)}\frac{\rho}{\mu}}{1-(1-h^*\sqrt{\frac{\mu}{L}})^{1/2}}.
\end{equation} 
\end{enumerate}
\end{theorem}

\begin{remark}[About the tracking error]
\label{remark_1}
The bounds in the assumptions for statements~\ref{tv-2},~\ref{tv-3} and~\ref{tv-4} of Theorem~\ref{th:tv} ensure that 
the rate at which the time-varying
optimization changes is bounded. Indeed, the right-hand sides of equations~\eqref{eq:asympt-cont},~\eqref{eq:asympt-dis} and~\eqref{eq:asympt-dis2} are consistent: the larger (lower) these bounds, the larger (lower) the asymptotic tracking error. Finally, and more importantly,  
the tracking is better when the contraction rate is larger.
\end{remark}

\section{Conclusion}
\label{sec:concl}

In this paper we first presented results on the discretization of systems of ODEs using contraction theory. Then, we proposed the Accelerated-Contracting-Nesterov flow 
and applied our methodology to the design of optimization algorithms based on implicit and explicit Euler discretizations. 
Finally, we extended our results to the case of time-varying optimization, where contracting properties of the system implies the tracking of the optimizer up to a bounded asymptotic error for both the optimization flow and algorithms.

An important future direction is to apply our discretization framework to the case where the objective function to minimize is convex and not necessarily strongly convex, and aim to establish the contracting nature of an associated optimization flow to then directly establish the convergence of its discretizations. We conjecture that such optimization flow  
will have a time-varying contraction rate. 
For example, it would be interesting to formulate a counterpart of the ACCONEST flow for the case where $f$ is not necessarily strongly convex --- perhaps resulting in a non-autonomous system --- and study it from a contraction perspective.

Another valuable future direction is to formulate a systematic mechanism where an optimization flows ODEs are built from an existing optimization algorithm, and then modify such ODEs to improve their convergence rate when discretized according to the contraction theory approach described in this paper.

The use of contraction theory also opens the possibility of other interesting future extensions such as the use of state-dependent metrics and/or semi-norms, and the formulation of families of contracting optimization algorithms according to these extensions.

\subsubsection*{Acknowledgments}
This work was funded in part by AFOSR award FA9550-22-1-0059. This work was mostly done while Pedro Cisneros-Velarde was a doctoral student at the University of California, Santa Barbara.

%
\bibliography{alias,Main,FB,Otr}

\begin{thebibliography}{}

\bibitem[Abraham et~al., 1988]{RA-JEM-TSR:88}
Abraham, R., Marsden, J.~E., and Ratiu, T.~S. (1988).
\newblock {\em Manifolds, Tensor Analysis, and Applications}, volume~75 of {\em
  Applied Mathematical Sciences}.
\newblock Springer, 2 edition.

\bibitem[Aminzare and Sontag, 2014]{ZA-EDS:14b}
Aminzare, Z. and Sontag, E.~D. (2014).
\newblock Contraction methods for nonlinear systems: {A} brief introduction and
  some open problems.
\newblock In {\em {IEEE} Conf.\ on Decision and Control}, pages 3835--3847.

\bibitem[Bastianello et~al., 2020]{NB-AS-RC:20}
Bastianello, N., Simonetto, A., and Carli, R. (2020).
\newblock Primal and dual prediction-correction methods for time-varying convex
  optimization.
\newblock {\em arXiv preprint arXiv:2004.11709}.

\bibitem[Boffi and Slotine, 2020]{NMB-JJES:20}
Boffi, N.~M. and Slotine, J.-J.~E. (2020).
\newblock A continuous-time analysis of distributed stochastic gradient.
\newblock {\em Neural Computation}, 32(1):36--96.

\bibitem[Cisneros-Velarde et~al., 2021]{PCV-SJ-FB:19r}
Cisneros-Velarde, P., Jafarpour, S., and Bullo, F. (2021).
\newblock Distributed and time-varying primal-dual dynamics via contraction
  analysis.
\newblock {\em IEEE Transactions on Automatic Control}.

\bibitem[Coppel, 1965]{WAC:1965}
Coppel, W.~A. (1965).
\newblock {\em Stability and Asymptotic Behavior of Differential Equations}.
\newblock Heath.

\bibitem[Davydov et~al., 2021]{AD-SJ-FB:20o}
Davydov, A., Jafarpour, S., and Bullo, F. (2021).
\newblock {Non-Euclidean} contraction theory for robust nonlinear stability.
\newblock {\em arXiv preprint arXiv:2103.12263}.

\bibitem[Desoer and Haneda, 1972]{CAD-HH:72}
Desoer, C.~A. and Haneda, H. (1972).
\newblock The measure of a matrix as a tool to analyze computer algorithms for
  circuit analysis.
\newblock {\em IEEE Transactions on Circuit Theory}, 19(5):480--486.

\bibitem[Fran{\c{c}}a et~al., 2020]{GF-JS-DPR-RV:20}
Fran{\c{c}}a, G., Sulam, J., Robinson, D.~P., and Vidal, R. (2020).
\newblock Conformal symplectic and relativistic optimization.
\newblock {\em Journal of Statistical Mechanics: Theory and Experiment}, 12.

\bibitem[Hairer et~al., 1993]{EH-SPN-GW:93}
Hairer, E., N\o{}rsett, S.~P., and Wanner, G. (1993).
\newblock {\em Solving Ordinary Differential Equations I. Nonstiff Problems}.
\newblock Springer.

\bibitem[Khalil, 2002]{HKK:02}
Khalil, H.~K. (2002).
\newblock {\em Nonlinear Systems}.
\newblock Prentice Hall, 3 edition.

\bibitem[Ladas and Lakshmikantham, 1972]{GEL-VL:72}
Ladas, G.~E. and Lakshmikantham, V. (1972).
\newblock {\em Differential Equations in Abstract Spaces}.
\newblock Academic Press.

\bibitem[Lohmiller and Slotine, 1998]{WL-JJES:98}
Lohmiller, W. and Slotine, J.-J.~E. (1998).
\newblock On contraction analysis for non-linear systems.
\newblock {\em Automatica}, 34(6):683--696.

\bibitem[Manchester and Slotine, 2017]{IRM-JJES:17b}
Manchester, I.~R. and Slotine, J.-J.~E. (2017).
\newblock On existence of separable contraction metrics for monotone nonlinear
  systems.
\newblock {\em IFAC-PapersOnLine}, 50(1):8226--8231.
\newblock 20th IFAC World Congress.

\bibitem[Muehlebach and Jordan, 2019]{MM-MJ:19}
Muehlebach, M. and Jordan, M.~I. (2019).
\newblock A dynamical systems perspective on {N}esterov acceleration.
\newblock In {\em International Conference on Machine Learning}, pages
  4656--4662.

\bibitem[Nesterov, 2018]{YN:18}
Nesterov, Y. (2018).
\newblock {\em Lectures on Convex Optimization}.
\newblock Springer, 2 edition.

\bibitem[{Nguyen} et~al., 2018]{HDN-TLV-KT-JJS:18}
{Nguyen}, H.~D., {Vu}, T.~L., {Turitsyn}, K., and {Slotine}, J.-J.~E. (2018).
\newblock Contraction and robustness of continuous time primal-dual dynamics.
\newblock {\em IEEE Control Systems Letters}, 2(4):755--760.

\bibitem[Revay and Manchester, 2020]{MR-IM:20}
Revay, M. and Manchester, I. (2020).
\newblock Contracting implicit recurrent neural networks: {Stable} models with
  improved trainability.
\newblock In {\em Conference on Learning for Dynamics and Control}, volume 120,
  pages 393--403.

\bibitem[Ryu and Boyd, 2016]{EKR-SB:16}
Ryu, E.~K. and Boyd, S. (2016).
\newblock Primer on monotone operator methods.
\newblock {\em Applied Computational Mathematics}, 15(1):3--43.

\bibitem[Shi et~al., 2021]{BS-SSD-MIJ-WJS:18}
Shi, B., Du, S.~S., Jordan, M.~I., and Su, W.~J. (2021).
\newblock Understanding the acceleration phenomenon via high-resolution
  differential equations.
\newblock {\em Mathematical Programming}.

\bibitem[Shi et~al., 2019]{BS-SSD-WS-MIJ:19}
Shi, B., Du, S.~S., Su, W., and Jordan, M.~I. (2019).
\newblock Acceleration via symplectic discretization of high-resolution
  differential equations.
\newblock In {\em Advances in Neural Information Processing Systems}.

\bibitem[Simonetto et~al., 2020]{AS-EDA-SP-GL-GBG:20}
Simonetto, A., {Dall'Anese}, E., Paternain, S., Leus, G., and Giannakis, G.~B.
  (2020).
\newblock Time-varying convex optimization: Time-structured algorithms and
  applications.
\newblock {\em Proceedings of the IEEE}, 108(11):2032--2048.

\bibitem[Su et~al., 2016]{WS-SB-EJC:16}
Su, W., Boyd, S., and Candes, E.~J. (2016).
\newblock A differential equation for modeling {Nesterov}'s accelerated
  gradient method: {Theory} and insights.
\newblock {\em Journal of Machine Learning Research}, 17:5312--5354.

\bibitem[Vaquero and Cortes, 2019]{MV-JC:19}
Vaquero, M. and Cortes, J. (2019).
\newblock Convergence-rate-matching discretization of accelerated optimization
  flows through opportunistic state-triggered control.
\newblock In {\em Advances in Neural Information Processing Systems}.

\bibitem[Vidyasagar, 2002]{MV:02}
Vidyasagar, M. (2002).
\newblock {\em Nonlinear Systems Analysis}.
\newblock SIAM.

\bibitem[Wensing and Slotine, 2020]{PMW-JJES:20}
Wensing, P.~M. and Slotine, J.-J.~E. (2020).
\newblock Beyond convexity --- {Contraction} and global convergence of gradient
  descent.
\newblock {\em PLoS One}, 15(8):1--29.

\bibitem[Wibisono et~al., 2016]{AW-ACW-MIJ:16}
Wibisono, A., Wilson, A.~C., and Jordan, M.~I. (2016).
\newblock A variational perspective on accelerated methods in optimization.
\newblock {\em Proceedings of the National Academy of Sciences},
  113(47):E7351--E7358.

\bibitem[Wilson et~al., 2018]{ACW-BR-MIJ:18}
Wilson, A.~C., Recht, B., and Jordan, M.~I. (2018).
\newblock A {Lyapunov} analysis of momentum methods in optimization.
\newblock {\em arXiv preprint arXiv:1611.02635}.

\bibitem[Zhang et~al., 2018]{JZ-AM-SS-AJ:18}
Zhang, J., Mokhtari, A., Sra, S., and Jadbabaie, A. (2018).
\newblock Direct {Runge}-{Kutta} discretization achieves acceleration.
\newblock In {\em Advances in Neural Information Processing Systems}.

\end{thebibliography}


\clearpage
\appendix

\thispagestyle{empty}

\section*{Appendix}

\section{Proof of Theorem~\ref{th:dis-contr}}
%
We first prove statement~(i). From the implicit Euler discretization we have $y_{k+1}-x^*=y_k-x^*+hg(y_{k+1},(k+1)h)$. Then,
\begin{align*}
\norm{y_{k+1}-x^*}_{2,P^{1/2}}^2&=\inpr{P^{1/2}(y_k-x^*)}{P^{1/2}(y_{k+1}-x^*)}+h\inpr{P^{1/2}g(y_{k+1},(k+1)h)}{P^{1/2}(y_{k+1}-x^*)}\\
&\leq \inpr{P^{1/2}(y_k-x^*)}{P^{1/2}(y_{k+1}-x^*)}-h\bar{\gamma}((k+1)h)\norm{y_{k+1}-x^*}_{2,P^{1/2}}^2
\end{align*}
where the inequality follows from the one-sided Lipschitz condition and $g(x^*,(k+1)h)=\zeros_n$. Then, since $\inpr{P^{1/2}(y_k-x^*)}{P^{1/2}(y_{k+1}-x^*)}\leq \norm{y_k-x^*}_{2,P^{1/2}}\norm{y_{k+1}-x^*}_{2,P^{1/2}}$, we obtain
\begin{equation}
\label{eq:one-step}
\norm{y_{k+1}-x^*}_{2,P^{1/2}}\leq (1+h\bar{\gamma}((k+1)h))^{-1}\norm{y_k-x^*}_{2,P^{1/2}}.
\end{equation}
Expanding the previous expression backwards in time leads to $\norm{y_{k+1}-x^*}_{2,P^{1/2}}\leq (1+h\bar{\gamma}((k+1)h))^{-1}\dots(1+h\bar{\gamma}((k-m)h))^{-1}\norm{y_{k-m-1}-x^*}_{2,P^{1/2}}$ for $m\geq 1$ and $k-m-1\geq 0$, 
from which equation~(1) from the main paper follows immediately. This finishes the proof of statement~(i).

We now prove statement~(ii). Since $g(x^*((k+1)h),(k+1)h)=\zeros_n$, we use~\eqref{eq:one-step} to obtain 
\begin{align*}
&\norm{y_{k+1}-x^*((k+1)h)}_{2,P^{1/2}}\\
&\leq(1+h\bar{\gamma}((k+1)h))^{-1}\norm{y_k-x^*((k+1)h)}_{2,P^{1/2}}\\
&\leq (1+h\bar{\gamma}((k+1)h))^{-1}(\norm{y_k-x^*(kh)}_{2,P^{1/2}}+\norm{x^*((k+1)h)-x^*(kh)}_{2,P^{1/2}}).
\end{align*}
Expanding the previous expression backwards in time leads to $\norm{y_{k+1}-x^*((k+1)h)}_{2,P^{1/2}}\leq (1+h\bar{\gamma}((k+1)h))^{-1}\dots(1+h\bar{\gamma}((k-m)h))^{-1}\norm{y_{k-m-1}-x^*}_{2,P^{1/2}}+(1+h\bar{\gamma}((k+1)h))^{-1}\dots(1+h\bar{\gamma}((k-m)h))^{-1}\norm{x^*((k-m)h)-x^*((k-(m+1))h)}_{2,P^{1/2}}+\dots+(1+h\bar{\gamma}((k+1)h))^{-1}\norm{x^*((k+1)h)-x^*(kh)}_{2,P^{1/2}}$ for $m\geq 1$, $k-m-1\geq 0$; 
from which equation~(3) from the main paper follows immediately. Now we prove the second part of statement~(ii). Observe that under the conditions of the statement, 
\begin{multline*}
\sum^k_{m=1}\prod^k_{\ell=m}(1+h\bar{\gamma}(mh))^{-1}\norm{x^*(mh)-x^*((m-1)h)}_{2,P^{1/2}}\leq\rho\sum^k_{m=1}\prod^k_{\ell=m}(1+h\gamma)^{-1}\\
=\rho\sum^k_{m=1}(1+h\gamma)^{-(k-m+1)}=\rho\sum^k_{m=1}(1+h\gamma)^{-m},
\end{multline*}
and this result along with the linear convergence result from statement~(i) imply equation~(4) from the main paper. Now, $\lim_{k\to\infty}\sum^k_{m=1}\frac{1}{(1+h\gamma)^{-m}}=\frac{1}{h\gamma}$ follows from the convergence of geometric series, and this along with $\lim_{k\to\infty}\frac{\norm{y_0-x^*(0)}_{2,P^{1/2}}}{(1+h\gamma)^k}=0$ imply equation~(5) from the main paper. This finishes the proof of statement~(ii).
%

\section{Proof of Theorem~\ref{th:dis-contr2}}

%
We first prove statement~(i). From the Euler discretization we have $y_{k+1}-x^*=y_k-x^*+hg(y_{k},kh)$. Then,
\begin{align*}
\norm{y_{k+1}-x^*}_{2,P^{1/2}}^2&=\norm{y_{k}-x^*+hg(y_k,kh)}_{2,P^{1/2}}^2\\
&=\norm{y_k-x^*}_{2,P^{1/2}}^2+2h\inpr{P^{1/2}(y_k-x^*)}{P^{1/2}g(y_k,kh)}+h^2\norm{g(y_k,kh)}_{2,P^{1/2}}^2\\
&\leq \norm{y_k-x^*}_{2,P^{1/2}}^2-2h\bar{\gamma}(kh)\norm{y_k-x^*}_{2,P^{1/2}}^2+h^2\norm{g(y_k,kh)}_{2,P^{1/2}}^2
\end{align*}
where the inequality follows from the one-sided Lipschitz condition. Now, we have 
\\$\norm{g(y_k,kh)}_{2,P^{1/2}}^2\leq\ell^2\norm{y_k-x^*}_2^2$ 
from $g$ being uniformly $\ell$-Lipschitz continuous and $g(x^*,kh)=\zeros_n$, 
and so 
\begin{align*}
\norm{y_{k+1}-x^*}_{2,P^{1/2}}&\leq(1-2h\bar{\gamma}(kh)+h^2\ell^2)^{1/2}\norm{y_k-x^*}_{2,P^{1/2}},
\end{align*}
and note that the theorem's conditions on the step-size $h$ ensures that $(1-2h\bar{\gamma}(kh)+h^2\ell^2)^{1/2}<1$. 
Expanding the previous expression backwards in time leads to $\norm{y_{k+1}-x^*}_{2,P^{1/2}}\leq (1-2h\bar{\gamma}(kh)+h^2\ell^2)^{1/2}\dots(1-2h\bar{\gamma}((k-m)h)+h^2\ell^2)^{1/2}\norm{y_{k-m}-x^*}_{2,P^{1/2}}$ for $m\geq 0$ and $k-m\geq 0$, 
from which equation~(6) from the main paper follows immediately. This finishes the proof of statement~(i).

We now prove statement~(ii). Note that, using the proof of statement~(i), $\norm{y_{k+1}-x^*(kh)}_{2,P^{1/2}}^2=\norm{y_k-x^*(kh)+hg(y_k,kh)}_{2,P^{1/2}}^2\leq(1-2h\bar{\gamma}(kh)+h^2\ell^2)\norm{y_k-x^*(kh)}_{2,P^{1/2}}$. Then, we obtain 
\begin{align*}
&\norm{y_{k+1}-x^*((k+1)h)}_{2,P^{1/2}}\\
&\leq \norm{y_{k+1}-x^*(kh)}_{2,P^{1/2}}+\norm{x^*((k+1)h)-x^*(kh)}_{2,P^{1/2}}\\
&\leq(1-2h\bar{\gamma}(kh)+h^2\ell^2)^{1/2}\norm{y_{k}-x^*(kh)}_{2,P^{1/2}}+\norm{x^*((k+1)h)-x^*(kh)}_{2,P^{1/2}}.
\end{align*}
Expanding the previous expression backwards in time leads to $\norm{y_{k+1}-x^*((k+1)h)}_{2,P^{1/2}}\leq (1-2h\bar{\gamma}(kh)+h^2\ell^2)^{1/2}\dots(1-2h\bar{\gamma}((k-m)h)+h^2\ell^2)^{1/2}\norm{y_{k-m}-x^*((k-m)h)}_{2,P^{1/2}}+(1-2h\bar{\gamma}(kh)+h^2\ell^2)^{1/2}\dots(1-2h\bar{\gamma}((k-(m-1))h)+h^2\ell^2)^{1/2}\norm{x^*((k-(m-1))h)-x^*((k-m)h)}_{2,P^{1/2}}+\dots+(1-2h\bar{\gamma}(kh)+h^2\ell^2)^{1/2}\norm{x^*(kh)-x^*((k-1)h)}_{2,P^{1/2}}+\norm{x^*((k+1)h)-x^*(kh)}_{2,P^{1/2}}$ for $m\geq 1$, $k-m-1\geq 0$; 
from which equation~(8) from the main paper follows immediately. The proof for the second part of statement~(ii) is very similar to the one done for the second part of statement~(ii) in Theorem~3.1, and thus is omitted. The proof of statement~(iii) is immediately obtained from $h^*=\arg\min_{h>0}-2h\gamma+h^2\ell^2$.
%

\section{Proof of Theorem~\ref{th:main-acc}}

Consider any $x_1,x_2,z_1,z_2\in\R^n$ and let 
$P=\begin{bmatrix}
a&b\\b&c
\end{bmatrix}\otimes I_n$ with $a,c>0$, $b\in\R$; and let $\beta:=\frac{\sqrt{\kappa}-1}{\sqrt{\kappa}+1}$. Then, the ACCONEST flow (equation~(15) from the main paper) is described by: 
$\dot{\bar{x}}_1=\bar{x}_2-\bar{x}_1-\frac{1}{L}\nabla f(\bar{x}_2)$ and 
$\dot{\bar{x}}_2= \beta(\bar{x}_2-\bar{x}_1)-\frac{\beta+1}{L}\nabla f(\bar{x}_2)$. 
After some algebraic work we obtain
\begin{align*}
\eta&:=(\fAcc(x_1,x_2)-\fAcc(z_1,z_2))^\top P 
\begin{bmatrix}
x_1-z_1\\
x_2-z_2
\end{bmatrix}\\
%
&=(a+b\beta)(x_1-z_1)^\top(x_2-z_2)-(a+b\beta)\norm{x_1-z_1}_2^2+\frac{a+b(\beta+1)}{L}(\nabla f(x_2)-\nabla f(z_2))^\top(x_1-z_1)\\
&\, +(b+c\beta)\norm{x_2-z_2}_2^2-(b+c\beta)(x_1-z_1)^\top(x_2-z_2)-\frac{b+c(\beta+1)}{L}(\nabla f(x_2)-\nabla f(z_2))^\top(x_2-z_2). 
\end{align*}
Now, we use the strong convexity property $-(\nabla f(x_2)-\nabla f(z_2))^\top(x_2-z_2)\leq -\mu\norm{x_2-z_2}_2^2$ and the $L$-smoothness of $f$ as $\norm{\nabla f(x_2)-\nabla f(z_2)}_2^2\leq L\norm{x_2-z_2}_2^2$ in 
the equality $\norm{\sqrt{\frac{\xi}{4L}}(\nabla f(x_2)-\nabla f(z_2))+\sqrt{\frac{L}{\xi}}(x_2-z_2)}_2^2=\frac{\xi}{4L}\norm{\nabla f(x_1)-\nabla f(z_1)}_2^2+(\nabla f(x_1)-\nabla f(z_1))^\top(x_2-z_2)+\frac{L}{\xi}\norm{x_2-z_2}_2^2$, with $\xi>0$, to obtain 
\begin{align*}
\eta&\leq (a+b\beta-(b+c\beta))(x_2-z_2)^\top(x_1-z_1)+\left(\frac{a+b(\beta+1)}{\xi}-a-b\beta\right)\norm{x_1-z_1}_2^2\\
&\quad +\left(\frac{a+b(\beta+1)}{4}\xi+(b+c\beta)-(b+c(\beta+1))\frac{1}{\kappa}\right)\norm{x_2-z_2}_2^2.
\end{align*}
Now, the coefficients multiplying the terms $\norm{x_1-z_1}_2^2$, $\norm{x_2-z_2}_2^2$ and $(x_1-z_1)^\top(x_2-z_2)$ must be less or equal than $-\frac{1}{\sqrt{\kappa}}a$, $-\frac{1}{\sqrt{\kappa}}c$ and $-\frac{2}{\sqrt{\kappa}}b$ respectively, in order 
that the vector field $\fAcc$ satisfies the one-sided Lipschitz condition with constant $-\frac{1}{\sqrt{\kappa}}$, i.e., with $\bar{\gamma}(t)\equiv \frac{1}{\sqrt{\kappa}}$ in the statement~(ii) of Theorem~2.1 and with $\norm{\cdot}_{2,P^{1/2}}$ on the right-hand side of its inequality. Thus, we need to analyze conditions on $a,b,c$ such that the following inequalities hold
\begin{align}
a\left(\frac{1}{\xi}-1+\frac{1}{\sqrt{\kappa}}\right)+b\left(\frac{\beta+1}{\xi}-\beta\right)&\leq 0\label{ineq-1}\\
\frac{a\xi}{4}+b\left(\frac{\xi}{4}(\beta+1)+1-\frac{1}{\kappa}\right)+c\left(\beta+\frac{1}{\sqrt{\kappa}}-\frac{\beta+1}{\kappa}\right)&\leq 0\label{ineq-2}\\
a+b\left(\beta-1+\frac{2}{\sqrt{\kappa}}\right)-c\beta&\leq 0.\label{ineq-3}
\end{align}

We first analyze conditions that ensure equation~\eqref{ineq-1} holds. We choose $\xi=\frac{\sqrt{\kappa}}{\sqrt{\kappa}-1}$ to eliminate the dependency on $a$ (recall that $\kappa>1$). With this value of $\xi$, the coefficient multiplying $b$ on the left-hand side of~\eqref{ineq-1} becomes 
$$
\frac{\sqrt{\kappa}-1}{\sqrt{\kappa}}(\beta+1)-\beta=\frac{\sqrt{\kappa}-1}{\sqrt{\kappa}+1}>0;
$$
thus to ensure~\eqref{ineq-1} holds, we choose $b<0$. For simplicity, in what follows in the proof, we do the change of variables $b\rightarrow-b$ so that we can take $b>0$. 

We now analyze conditions that ensure equation~\eqref{ineq-2} holds, which after the change of variables $b\rightarrow-b$ and rearranging some terms becomes 
\begin{equation}
\label{eq:clon-ineq2}
\frac{\xi}{4}\left(a-b(1+\beta)\right)-b\left(1-\frac{1}{\kappa}\right)+c\left(\beta+\frac{1}{\sqrt{\kappa}}-\frac{\beta+1}{\kappa}\right)\leq 0.
\end{equation}
Now, by establishing the condition 
\begin{equation}
\label{cond2-proof}
a\leq b(1+\beta)
\end{equation}
and observing that $\beta+\frac{1}{\sqrt{\kappa}}-\frac{1+\beta}{\kappa}=\frac{\kappa-1}{\sqrt{\kappa}(\sqrt{\kappa}+1)}=\left(1-\frac{1}{\kappa}\right)\left(\frac{\sqrt{\kappa}}{\sqrt{\kappa}+1}\right)<1-\frac{1}{\kappa}$, we conclude that in order to satisfy inequality~\eqref{eq:clon-ineq2}, we need to ensure:
\begin{equation}
\label{cond3-proof}
-b\left(1-\frac{1}{\kappa}\right)+c\left(1-\frac{1}{\kappa}\right)\left(\frac{\sqrt{\kappa}}{\sqrt{\kappa}+1}\right)\leq 0\quad\Longrightarrow \quad c\leq \frac{\sqrt{\kappa}+1}{\sqrt{\kappa}}b.
\end{equation}

Now, recall that we need $P$ to be positive definite, for which the inequality $ac>b^2$ must hold. Therefore, to ensure $P$ is positive definite, we take $c=\frac{\sqrt{\kappa}+1}{\sqrt{\kappa}}b$, which satisfies~\eqref{cond3-proof}, and choose $a=\gamma\frac{\sqrt{\kappa}}{\sqrt{\kappa}+1}b$ where $\gamma>1$ must be chosen such that~\eqref{cond2-proof} is satisfied.

We now analyze conditions that ensure equation~\eqref{ineq-3} holds, which after the change of variables $b\rightarrow-b$ becomes 
\begin{equation}
\label{eq:clon-ineq3}
a-b\left(\beta - 1+\frac{2}{\sqrt{\kappa}}\right)-c\beta\leq 0.
\end{equation}
Replacing our newly assigned values for $a$ and $c$ in~\eqref{eq:clon-ineq2}, we only need to verify that the following inequality holds
\begin{equation*}
\gamma\frac{\sqrt{\kappa}}{\sqrt{\kappa}+1}-\left(\beta-1+\frac{2}{\sqrt{2}}\right)-\frac{\sqrt{\kappa}+1}{\sqrt{\kappa}}\beta\leq 0
\end{equation*}
for an appropriate value of $\gamma>0$. After some algebraic work, we conclude that $\gamma\leq 1+\frac{1}{\kappa}$ satisfies this inequality, besides of defining a value for $a$ which 
satisfies inequality~\eqref{cond2-proof}.

Finally, after the change of variables $b\rightarrow-b$, we set $b=1$ and conclude that the ACCONEST flow 
satisfies 
the one-sided Lipschitz condition of Theorem~2.1 
with constant $-\frac{1}{\sqrt{\kappa}}$. Thus, contraction of the system in equation (15) from the main paper follows as stipulated in statement~(i). The proof of statement~(ii) follows immediately from Theorem~2.1.
  
  
Now we prove statement~(iii). First, we use the result in statement~(i) of Theorem~3.1 to obtain 
\begin{multline}
\norm{y_k^{(1)}-x^*}_2^2\leq \norm{y_k^{(1)}-x^*}_2^2+\norm{y_k^{(2)}-x^*}_2^2\\
\leq \frac{1}{\lambda_{\max}(P)}\norm{(y_k^{(1)}-x^*,y_k^{(2)}-x^*)^\top}_{2,P^{1/2}}^2\leq(1+h\frac{1}{\sqrt{\kappa}})^{-k}\norm{(y_0^{(1)}-x^*,y_0^{(2)}-x^*)^\top}_{2,P^{1/2}}^2.
\end{multline}
%
Using this result, along with 
%
$f(y^{(1)}_k)-f(x^*)\leq \nabla f(x^*)^\top(y^{(1)}_k-x^*)+\frac{L}{2}\norm{y^{(1)}_k-x^*}_2^2=\frac{L}{2}\norm{y^{(1)}_k-x^*}_2^2$ from the $L$-smoothness of $f$, we obtain
$$
f(y^{(1)}_k)-f(x^*)\leq \frac{L}{2\lambda_{\min}(P)}\norm{(y_0^{(1)}-x^*,y_0^{(2)}-x^*)^\top}_{2,P^{1/2}}^2\left(1+h\frac{1}{\sqrt{\kappa}}\right)^{-2k},
$$ 
which concludes the proof of statement~(iii).
%

  
Now we prove statement~(iv). In order to apply Theorem~3.4, we need to find a Lipschitz constant for the vector field $\fAcc$. Thus,
\begin{equation}
\label{eq:up_des_t1}
\begin{aligned}
&\norm{\fAcc(y^{(1)}_k,y^{(2)}_k)}_{2,P^{1/2}}^2\\
&\leq \lambda_{\max}(P)\norm{\fAcc(y^{(1)}_k,y^{(2)}_k)}_2^2\\
&\leq \lambda_{\max}(P)(7+5\beta+6\beta^2)\norm{(y^{(1)}_k-x^*,y^{(2)}_k-x^*)^\top}_2^2\\
&\leq \frac{\lambda_{\max}(P)}{\lambda_{\min}(P)}(7+5\beta+6\beta^2)\norm{(y^{(1)}_k-x^*,y^{(2)}_k-x^*)^\top}_{2,P^{1/2}}^2.
\end{aligned}
\end{equation}
We now explain how we obtain the second inequality in~\eqref{eq:up_des_t1}. Firstly, note that $\norm{\fAcc(y^{(1)}_k,y^{(2)}_k)}_2^2=\norm{y^{(2)}_k-y^{(1)}_k-\frac{1}{L}\nabla f(y^{(2)}_k)}_2^2+\norm{\beta(y^{(2)}_k-y^{(1)}_k)-\frac{1+\beta}{L}\nabla f(y^{(2)}_k)}_2^2$.
Now, $\norm{y^{(2)}_k-y^{(1)}_k-\frac{1}{L}\nabla f(y^{(2)}_k)}_2\leq\norm{y^{(2)}_k-x^*}_2+\norm{y^{(1)}_k-x^*}_2+\frac{1}{L}\norm{\nabla f(y^{(2)}_k)- \nabla f(x^*)}_2\leq 2\norm{y^{(2)}_k-x^*}_2+\norm{y^{(1)}_k-x^*}_2$ by using the triangle inequality and the $L$-smoothness of $f$. Then, $\norm{y^{(2)}_k-y^{(1)}_k-\frac{1}{L}\nabla f(y^{(2)}_k)}_2^2\leq 4\norm{y^{(2)}_k-x^*}_2^2+\norm{y^{(1)}_k-x^*}_2^2+4\norm{y^{(2)}_k-x^*}_2\norm{y^{(1)}_k-x^*}_2\leq 6\norm{y^{(2)}_k-x^*}_2^2+3\norm{y^{(1)}_k-x^*}_2^2$, where the last inequality follows from $a^2+b^2\geq 2ab$ for any $a,b\in\R$. 
Likewise, we can obtain $\norm{\beta(y^{(2)}_k-y^{(1)}_k)-\frac{1+\beta}{L}\nabla f(y^{(2)}_k)}_2^2\leq (1+4\beta+6\beta^2)\norm{y^{(2)}_k-x^*}_2^2+(\beta+3\beta^2)\norm{y^{(1)}_k-x^*}_2^2$. Putting it all together, we obtain the sought inequality: $\norm{\fAcc(y^{(1)}_k,y^{(2)}_k)}_2^2=\norm{y^{(2)}_k-y^{(1)}_k-\frac{1}{L}\nabla f(y^{(2)}_k)}_2^2+\norm{\beta(y^{(2)}_k-y^{(1)}_k)-\frac{1+\beta}{L}\nabla f(y^{(2)}_k)}_2^2\leq (7+5\beta+6\beta^2)\norm{(y^{(1)}_k-x^*,y^{(2)}_k-x^*)^\top}_2^2$.

Then, since the square of the Lipschitz constant is $\frac{\lambda_{\max}(P)}{\lambda_{\min}(P)}(7+5\beta+6\beta^2)$, the optimal integration step-size according to Theorem~3.4 is  $h^*:=\frac{1}{\sqrt{\kappa}(7+5\beta+6\beta^2)}\frac{\lambda_{\min}(P)}{\lambda_{\max}(P)}$, and then we have that
\begin{equation*}
\norm{(y^{(1)}_k-x^*,y^{(2)}_k-x^*)^\top}_{2,P^{1/2}}\leq (1-h^*\sqrt{\frac{\mu}{L}})^{k/2}\norm{(y^{(1)}_0-x^*,y^{(2)}_0-x^*)^\top}_{2,P^{1/2}} 
\end{equation*}
for $k=1,2,\dots$. 
Finally, a similar analysis to the proof of statement~(iii) leads to 
$$
f(y^{(1)}_k)-f(x^*)\leq \frac{L}{2\lambda_{\min}(P)}\norm{(y_0^{(1)}-x^*,y_0^{(2)}-x^*)^\top}_{2,P^{1/2}}^2\left(1-h^*\sqrt{\frac{\mu}{L}}\right)^{k},
$$ 
which concludes the proof of statement~(iv).
%
%

\section{Proof of Theorem~\ref{th:tv}}

The proof of statement~(i) is virtually the same as the one of statement~(i) of Theorem~4.1 due to the assumption $f(\cdot,t)\in\SuL$ for every $t\geq 0$. 
Now we prove statement~(ii). Let us fix any $t\geq 0$ and observe that 
\begin{equation}
\label{aca1}\zeros_n=-\nabla f(x^*(t),t).
\end{equation}
We first show that the curve $t\mapsto x^*(t)$ is continuously differentiable. Define the function
$g:\R^{n+1}\to \R^{n}$ as $g(t,x) = -\nabla f(x,t)$. 
Since $t\mapsto \nabla f(x,t)$ is continuously
differentiable, the function $g$ is continuously differentiable on
$\R^{n+1}$. Moreover, note that 
$\nabla_{x} g(t,x) = -\nabla^2 f(x,t)\preceq-\mu I_n$ which implies that $\nabla_{x} g(t,x)$ 
is Hurwitz and therefore nonsingular. Then, the Implicit Function
\citep[Theorem 2.5.7]{RA-JEM-TSR:88} implies the solutions $t\mapsto x^*(t)$ of the algebraic equation~\eqref{aca1} is continuously differentiable for any $t\geq 0$. 

Now, differentiating equation~\eqref{aca1} with respect to time,
\begin{align*}
&\implies \zeros_{n}=-\nabla^2f(x^*(t),t)\dot{x^*}(t)-\dot{\nabla}f(x^*(t),t)\\
&\implies \zeros_{m}=-\dot{x^*}(t)-(\nabla^2f(x^*(t),t))^{-1}\dot{\nabla}f(x^*(t),t)\\
&\implies \norm{\dot{x^*}(t)}_2=\norm{(\nabla^2f(x^*(t),t))^{-1}\dot{\nabla}f(x^*(t),t)}_2\leq
\norm{(\nabla^2f(x^*(t),t))^{-1}}_2\norm{\dot{\nabla}f(x^*(t),t)}_2\leq
\frac{\rho}{\mu}
\end{align*}
where the last inequality follows from $\mu I_n\preceq\nabla^2 f(x,t)\preceq L I_n\implies \frac{1}{L} I_n\preceq(\nabla^2 f(x,t))^{-1}\preceq \frac{1}{\mu} I_n$ and thus $\norm{(\nabla^2 f(x,t))^{-1}}_2\leq\frac{1}{\mu}$ for every $(t,x)\in[0,\infty)\times{\R^n}$.

Now, considering the contraction of the system of equation~(20) from the main paper from statement~(i), we set\\ $\delta(t):=\left\|\begin{bmatrix}x_1(t)\\ x_2(t)\end{bmatrix}-\begin{bmatrix}x^*(t)\\ x^*(t)\end{bmatrix}\right\|_{2,P^{1/2}}$ and use~\citep[Lemma~2]{HDN-TLV-KT-JJS:18} to obtain the following differential inequality $\dot{\delta}(t)\leq -\sqrt{\frac{\mu}{L}}\delta(t)+\left\|\begin{bmatrix}\dot{x^*}(t)\\ \dot{x^*}(t)\end{bmatrix}\right\|_{2,P^{1/2}}$. Then, using our previous results, 
$$\dot{\delta}(t)\leq -\sqrt{\frac{\mu}{L}}\delta(t)+\sqrt{2\lambda_{\max}(P)}\norm{\dot{x^*}}_{2}\leq -\sqrt{\frac{\mu}{L}}\delta(t)+\sqrt{2\lambda_{\max}(P)}\frac{\rho}{\mu}.$$ 
Now, since the function $h(u)=-\sqrt{\frac{\mu}{L}}u+\sqrt{2\lambda_{\max}(P)}\frac{\rho}{\mu}$ is Lipschitz (it is an affine function), we use the Comparison Lemma~\citep{HKK:02} to 
upper bound $\delta(t)$ by the solution to the differential equation $\dot{u}(t)=-\sqrt{\frac{\mu}{L}}u+\sqrt{2\lambda_{\max}(P)}\frac{\rho}{\mu}$ for all $t\geq 0$, from which statement~(ii) follows. 

Now we prove both statements~(iii) and~(iv). First, notice that 
$
\norm{x^*(kh)-x^*((k-1)h)}_2\leq\frac{1}{\mu}\norm{\nabla f(x^*(kh),kh)-\nabla f(x^*((k-1)h),kh)}_2=\frac{1}{\mu}\norm{\nabla f(x^*((k-1)h),kh)}_2=\frac{1}{\mu}\norm{\nabla f(x^*((k-1)h),kh)-\nabla f(x^*((k-1)h),(k-1)h)}_2\leq \frac{\rho}{\mu}
$,
where the first inequality follows from the strong convexity assumption. Using this result, we obtain 
\begin{equation}
\label{eq:auxin1}
\left\|\begin{bmatrix}x^*(kh)\\ x^*(kh)\end{bmatrix}-\begin{bmatrix}x^*((k-1)h)\\ x^*((k-1)h)\end{bmatrix}\right\|_{2,P^{1/2}}=\sqrt{2\lambda_{\max}(P)}\norm{x^*(kh)-x^*((k-1)h)}_2\leq \sqrt{2\lambda_{\max}(P)}\frac{\rho}{\mu}.
\end{equation}
Since the system is contracting from statement~(i), we use~\eqref{eq:auxin1} and statement~(ii) of Theorem~3.1 to conclude the proof of statement~(iii). Finally, the proof of statement~(iv) is very similar: again, since the system is contracting, we use~\eqref{eq:auxin1} and statement~(ii) of Theorem~3.4 to conclude the proof.

\end{document}